\documentclass[12pt,a4paper]{article}

\usepackage[T1]{fontenc}
\usepackage[utf8]{inputenc}
\usepackage[a4paper,margin=1in]{geometry}
\usepackage{microtype}
\usepackage{parskip}
\usepackage{enumitem}
\usepackage{lmodern}
\usepackage{graphicx}

\usepackage{xcolor}

\usepackage{amsmath,amssymb,amsfonts,amsthm,mathtools}
\usepackage{mathrsfs}
\usepackage{bm}

\usepackage[colorlinks=true,linkcolor=blue,citecolor=blue,urlcolor=blue]{hyperref}

\newtheorem{introtheorem}{Theorem}

\newtheorem{conjecture}{Conjecture}

\newtheorem{introproposition}{Proposition}

\newtheorem{theorem}{Theorem}[section]
\newtheorem{lemma}[theorem]{Lemma}
\newtheorem{proposition}[theorem]{Proposition}

\theoremstyle{definition}

\theoremstyle{remark}

\newcommand{\R}{\mathbb{R}}
\newcommand{\C}{\mathbb{C}}

\newcommand{\Sp}{\mathbb{S}}

\DeclareMathOperator{\tr}{tr}

\newcommand{\abs}[1]{\left|#1\right|}
\newcommand{\norm}[1]{\left\lVert#1\right\rVert}

\newcommand{\angles}[1]{\left\langle#1\right\rangle}

\title{Polarization problems and Coxeter systems}

\author{
Ángel D. Martínez
\and
Oscar Ortega-Moreno
}

\begin{document}

	\maketitle

% Uncomment the next line if you want a table of contents.
% \tableofcontents

	\begin{abstract}
		In this paper we provide a proof of the strong polarization conjecture due to Ball and Frenkel. We observe that Coxeter systems produce unexpected new examples of extremizers. Furthermore, as a by-product of our work we solve the real polarization problem and completely characterize its extremal cases.	
	\end{abstract}

% ============================================================
% Introduction
% ============================================================
\section{Introduction}

	Twenty years ago the following conjecture was raised by Ball and Frenkel:
	\begin{conjecture}[Strong polarization conjecture]\label{strong}
		Given a set of $n$ vectors $v_1,\dots,v_n~\in~\mathbb{S}^{d-1}$, there exists a unit vector $u\in\mathbb{S}^{d-1}$ such that
		$$
		\sum_{j=1}^n\frac{1}{\langle v_j, u\rangle^2}\leq n^2.
		$$
	\end{conjecture}
	Prior to this work, not much was known about this conjecture, with the exception of the planar case proved by Ambrus (cf. \cite{Ambrus, Ambrus2013}). The real polarization problem, which had originated a decade earlier, follows from the strong polarization conjecture. We postpone a detailed discussion of the former and turn first to the extremal cases of Conjecture~\ref{strong}.

	Ball and Frenkel observed that orthogonal bases are not the only extremizers for the strong polarization problem: that evenly distributed points on a great circle give the same bound. This observation suggested that the problem captures geometric structure beyond that exhibited by orthogonal configurations alone. In fact, as we shall explain, there are many other extremal sets, all arising from an unforeseen connection to Coxeter systems.

	For any nonzero vector $v\in \mathbb{R}^d$, let $s_v$ denote the reflection across the orthogonal hyperplane, namely $v^\perp$. A \emph{finite reflection system of unit vectors} in $\mathbb{R}^d$ is a finite subset $\Phi \subset \Sp^{d-1}$ which is invariant under reflections (i.e. such that, for every $v \in \Phi$, one has $s_v(\Phi)=\Phi$).

	\begin{introproposition}\label{strongextremal}
		Let $v_1,\dots,v_n \in \Sp^{d-1}$ be unit vectors, no two of which are parallel, and suppose that
		$$
		\Phi=\{\pm v_1,\dots,\pm v_n\}
		$$
		is a finite reflection system. Then, $v_1,\dots,v_n$ is an extremal configuration of Conjecture $\ref{strong}$.
	\end{introproposition}

	Thus, in dimension $d$, this result yields extremal configurations associated with all finite Coxeter systems of total rank $d$. We note that reflection systems have appeared previously in the study of polarization problems, for instance in the work of Leung, Li, and Rakesh \cite{Leung2008}, where they served as examples for which the polarization constant could be computed explicitly. Their role here is different: finite reflection arrangements give rise to a general class of extremal configurations for Conjecture~\ref{strong}. It is natural to ask whether the extremizers are exhausted by the configurations arising from such systems. In dimension two, these configurations are the reflection arrangements of dihedral groups, namely the symmetry groups of regular polygons; in dimension three, they arise from the rank-$3$ finite reflection arrangements, including those associated with regular prisms and Platonic solids (see Figure~\ref{fig:coxeter-gallery} below).
	
	\begin{figure}[h]
		\centering
		\includegraphics[width=0.5\textwidth]{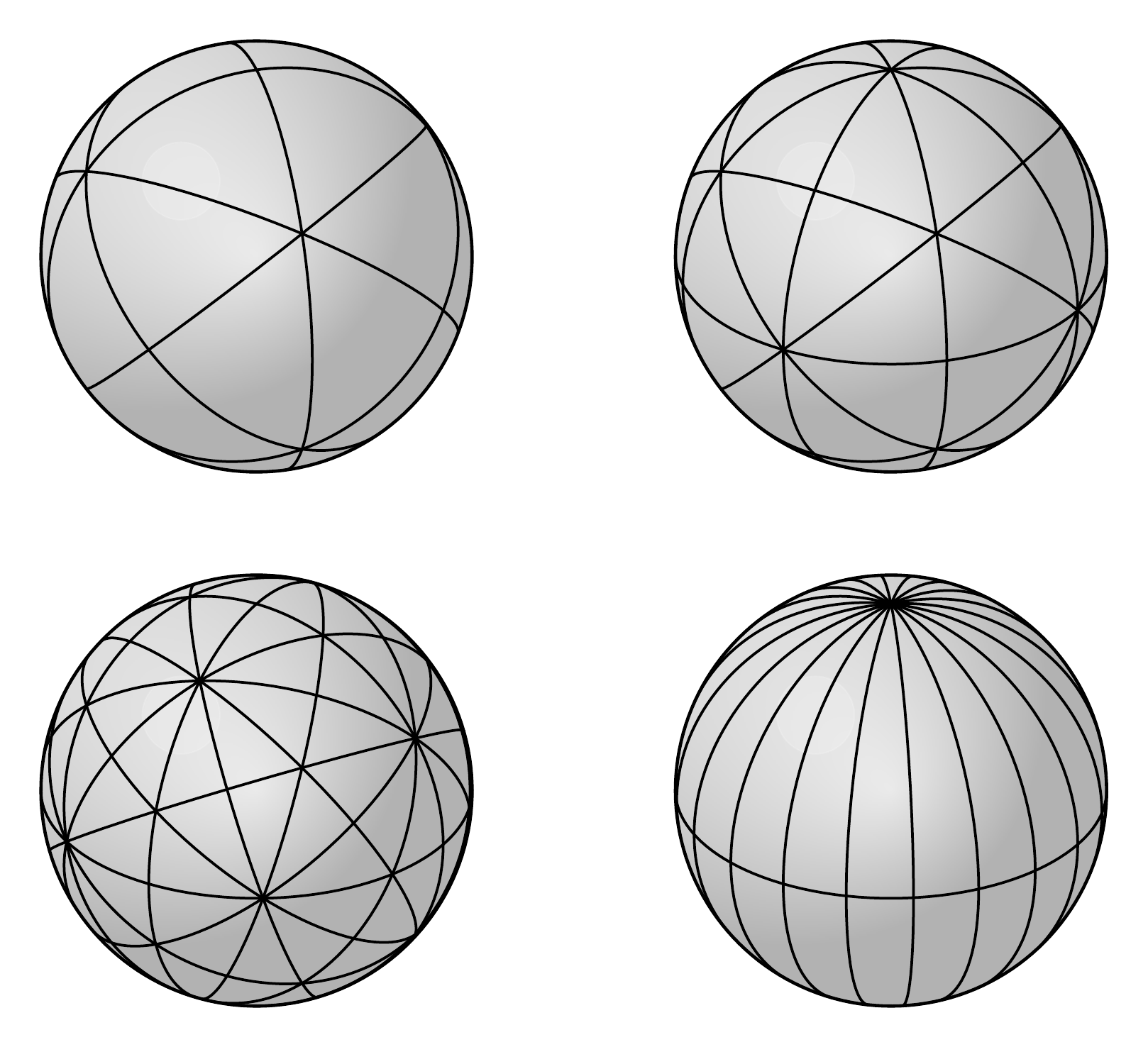}
		\caption{\small Equality cases for the strong polarization problem in $\mathbb{R}^3$ arising from rank-$3$ Coxeter arrangements on $\Sp^2$. The curves are the great-circle cuts induced by the reflecting hyperplanes. Here $A_3$, $B_3$, $H_3$ and $A_1\times I_2(10)$ are the arrangements associated with the tetrahedron, cube, dodecahedron or a prism with a regular decagon base, respectively. Notice that the arrangements associated with the octahedron and icosahedron are $B_3$ and $H_3$ as they are dual to a cube and dodecahedron, respectively.}
		\label{fig:coxeter-gallery}
	\end{figure}

	We now turn to the origins of polarization problems.

% ------------------------------------------------------------
% The polarization problem
% ------------------------------------------------------------
\subsection*{The polarization problem}

	The polarization problem emerged from the study of multilinear operators on Banach spaces. A fundamental geometric consequence of the Hahn--Banach theorem is that every vector in a Banach space can be normed by a continuous linear functional: for every $x\in X$, there exists $f\in X^*$ with $\norm{f}=1$ satisfying $f(x)=\norm{x}.$ The picture becomes far more complex when several linear functionals are considered. Given $f_1,\dots,f_n\in X^*$, these functionals may not attain their norm at the same point simultaneously. It is therefore natural to measure this discrepancy by asking how large the product $f_1(x)\cdots f_n(x)$ can be on the unit ball of $X$. This problem depends, of course, on the geometry of the space $X$ in question. Motivated by the problem of quantifying this phenomenon, the polarization constants of Banach spaces were introduced. The $n$th polarization constant of $X$, denoted by $c_n(X)$, is the smallest
	constant $C>0$ such that
	$$
	\norm{f_1}\cdots\norm{f_n}\leq C \norm{f_1\cdots f_n}
	$$
	for every $f_1,\dots,f_n\in X^*$, where
	$$
	\norm{f_1\cdots f_n} = \sup_{\norm{x}\leq 1}\abs{f_1(x)\cdots f_n(x)}.
	$$
	It is a known consequence of Ball's plank theorem \cite{Ball1991} that there is a universal bound $c_n(X)\leq n^n$ for every Banach space $X$ and every $n\geq 1$. Moreover, this estimate is sharp in general. Indeed, considering $X=\ell_1^n$ and the coordinate functionals $f_j(x)=x_j$, one has $\norm{f_j}=1$ for every $j$, while
	$$
	\norm{f_1\cdots f_n} = \sup_{\norm{x}_{\ell_1^n}\leq 1}|x_1\cdots x_n| = \frac{1}{n^n}.
	$$
	The last equality follows from the Arithmetic--Geometric Mean inequality, with equality at $x=(1/n,\dots,1/n)$. This shows that the universal estimate is best possible in the class of Banach spaces. As observed by Benítez, Sarantopoulos and Tonge \cite{Benitez1998}, an application of John’s theorem on maximal-volume ellipsoids shows that $c_n(\ell_2^n)\leq c_n(X)\leq n^{n/2}c_n(\ell_2^n)$ for any $n$-dimensional Banach space $X$ \cite{John1948} (see also \cite{Ball1992}). This connection suggests the study of the polarization problem for Hilbert spaces only.

	In the Hilbert space case the problem turns out to be far more subtle. A conjecture often attributed to Benítez, Sarantopoulos and Tonge asks whether for any Hilbert space $\mathcal{H}$, the sharper estimate
	$$
	c_n(\mathcal{H})\leq n^{n/2}
	$$
	holds \cite{Benitez1998}, with equality if $\dim(\mathcal{H})= n$. They claimed its truth for any $n\leq 4$. To support this conjecture, Pappas and Révész
	verified it for $n \leq 5$ in \cite{Pappas2004}. We will refer to it in the following equivalent form:

	\begin{conjecture}[Real polarization problem \cite{Benitez1998}]\label{weak}
		Given unit vectors $v_1,\dots,v_n \in \mathbb{S}^{d-1}$, there exists a unit
		vector $u\in\mathbb{S}^{d-1}$ such that
		\begin{equation*}
			\prod_{j=1}^n|\langle v_j,u\rangle|\geq n^{-n/2}.
		\end{equation*}
		Moreover, $\sup_{\norm{x}=1}\abs{\langle v_1,x\rangle\cdots\langle v_n,x\rangle} = n^{-n/2}$ if and only if $d \geq n$ and $v_1,\dots,v_n$ form an orthonormal set of $\R^d$.
	\end{conjecture}

	In the complex case, this question was settled by Arias de Reyna \cite{AriasdeReyna1998}. His proof relies on the relation between complex Gaussian moments and permanents, together with Lieb’s permanent inequality for positive semidefinite matrices \cite{Lieb1966}. In \cite{BALL2001} Ball later gave a different proof of this result as a consequence of his solution
	of the complex plank problem for Hilbert spaces (see
	\cite{OrtegaMoreno2022} for a streamlined version of Ball's argument). A proof of the real case has been elusive for about three decades, despite the fact that the correct asymptotic behavior was known.

	It should be noted that Conjecture~\ref{weak} follows from the resolution of the stronger Conjecture \ref{strong}, except for the characterization of extremal cases (see Lemma~\ref{lem:AGM} in the Appendix). Before proceeding, we briefly review the state of the art. To the best of our knowledge, the first estimate for products of linear forms over the reals appeared in the sixties in the work of Marcus and Minc \cite{Marcus1961}, in connection with their study of permanent inequalities. Marcus later returned to this question, obtaining an eigenvalue-dependent lower bound \cite{Marcus1997}. It also appears in the work of Ryan and Turett \cite{Ryan1998}. In both works the problem seems to have been brought to their attention by Sarantopoulos.

	Several authors have since contributed to the real polarization problem, either by proving the conjecture in special cases or by obtaining general upper bounds for $c_n(\mathbb{R}^n)$. One line of results seeks bounds of the form $c_n(\mathbb R^n)\leq (C n)^{n/2}$, with different absolute constants $C$. Litvak, Milman and Schechtman obtained such a bound with $C = 2e^{\gamma}\approx 3.56$ (where $\gamma$ denotes the Euler--Mascheroni constant). This was refined by García-Vázquez and Villa, who obtained sharper finite-dimensional estimates with the same asymptotic behavior. Révész and Sarantopoulos obtained the estimate $c_n(\mathbb R^n)\leq 2^{n/2-1} n^{n/2}$, which Frenkel improved to $c_n(\mathbb R^n)\leq (3^{3/2}e^{-1} n)^{n/2}$. Finally, Muñoz-Fernández, Sarantopoulos and Seoane-Sepúlveda proved the asymptotically sharper bound $c_n(\mathbb R^n)\leq n2^{n/4}n^{n/2}$. For more details we refer the reader to \cite{Litvak1998, GarciaVazquez1999,Revesz2004,Frenkel2008} and \cite{MunozFernandez2010}, respectively.

	This problem has already been approached from several perspectives. A particularly natural one comes from the connection between polarization constants and plank problems. Motivated by Bang's method in his solution to the plank problem of Tarski \cite{Bang1951}, one may try to choose signs $\varepsilon_j\in\{-1,1\}$ so that
	$$
	v=\sum_{j=1}^n \varepsilon_j v_j
	$$
	has maximal norm, and then test the conjectured inequality at the unit vector $u=v/\norm{v}.$ This vector is often called the normalized longest signed sum, or Bang's vector,
	associated with the system. In \cite{Pinasco2023}, Pinasco gave a {\em tour de force} proof showing that this choice indeed proves the conjecture for all $n\leq 14$. However, this approach cannot fully settle the conjecture: Matolcsi and Muñoz constructed an example for $n=34$ in which the normalized longest signed sum fails to satisfy the desired lower bound \cite{Matolcsi2006}. As a matter of fact, we prove that the conjectured inequality holds at a different point on the sphere.

	Before we continue, let us introduce some notation we will use in the sequel: given vectors $v_1,\ldots,v_n\in\mathbb{R}^d$, we consider the polynomial
	$$
	P(x)=\prod_{j=1}^n \langle v_j,x\rangle,
	$$
	for any $x\in\mathbb{R}^d$. We define $\mathscr{E}(P)$ to be the set of {\em local extrema} of the polynomial $P$ considered as a function on the sphere $\Sp^{d-1}$, namely those points $u\in\Sp^{d-1}$ such that $\nabla P(u)$ is proportional to $u$ and $P(u)\neq 0$. The last condition distinguishes these points from the {\em larger} set of critical points.

	A possible approach to solving the strong polarization conjecture is motivated by Ball’s solution of the complex plank problem \cite{Ball1991}: one may optimize $P$ on the sphere and attempt to prove the desired inequality at the maximizer. In the planar case this was developed by Ambrus in \cite{Ambrus}. There he reformulated the problem in terms of inverse eigenvectors of Gram matrices (see also \cite{Ambrus2013}). Let us mention in passing that this variational approach was later used by the second author to give an alternative proof of Fejes Tóth's conjecture \cite{OrtegaMoreno}, originally proved by Jiang and Polyanskii \cite{Jiang2017}. Zhao subsequently streamlined the argument by avoiding Gram matrices \cite{Zhao2022}.

	In the present paper, we go one step further. Rather than looking for a single distinguished optimizer, we consider the set of {\em all} extremal points on the sphere simultaneously and show that at least one of them satisfies the desired inequality. The key observation is the existence of a positive-weight function on the set of extremal points such that, when certain quantities are averaged with respect to these weights, the desired conclusion follows. We now state our main result:

	\begin{introtheorem}
		\label{main}
		Let $v_1,\ldots,v_n$ be unit vectors in $\mathbb{R}^d$, no two of which are parallel. Then, there exists a weight function $\mu:\mathscr{E}(P)\rightarrow(0,\infty)$ such that
		$$
		\sum_{u\in \mathscr{E}(P)}\left(\sum_{j=1}^n\frac{1}{\langle v_j, u\rangle^2}-n^2\right)\mu(u)= 0.
		$$
		Furthermore,
		$$
		\mu(u)=\det\left(I+\frac{1}{n}\sum_{j=1}^n\frac{v_j\otimes v_j}{\langle v_j,u\rangle^2}\right)^{-1}.
		$$
	\end{introtheorem}
	As an immediate consequence, there exists a local extremum $u\in\mathscr{E}(P)$ such that
	$$
	\sum_{j=1}^n \frac{1}{\langle v_j,u\rangle^2}\leq n^2.
	$$
	The proof of Theorem~\ref{main} relies on a delicate application of a deep result in algebraic geometry, the Euler–Jacobi vanishing theorem (see Lemma~\ref{eulerjacobi}). Roughly speaking, this theorem states that, for a finite system of polynomial equations with well-behaved isolated solutions, a certain weighted sum of the values of a polynomial over those solutions is automatically zero whenever the polynomial has sufficiently low degree.

	This result shows that Conjecture~\ref{strong} holds at this critical point. Consequently, Conjecture~\ref{weak} follows \emph{a posteriori}, except for the characterization of extremal cases (see Lemma~\ref{lem:AGM} in the Appendix).

	A system of vectors $v_1,\dots,v_n\in \Sp^{d-1}$ is an {\em extremal configuration} of Conjecture~\ref{weak} if the associated polynomial $P$ attains the conjectured bound, namely
	$$
	\sup_{u\in \Sp^{d-1}}|P(u)| = n^{-\frac{n}{2}}.
	$$
	As an application of Theorem~\ref{main}, we also resolve the classification of extremal configurations of Conjecture~\ref{weak}.
	\begin{introtheorem}\label{weakextremal}
		Let $v_1,\dots,v_n\in \Sp^{d-1}$ be an extremal configuration of Conjecture~\ref{weak}. Then, $d\geq n$ and $v_1,\dots,v_n$ form an orthonormal set.
	\end{introtheorem}

	This together with our previous considerations settles Conjecture~\ref{weak} completely.

	The article is organized as follows: in Section \ref{sec:extrema} we study the set of local extrema of $P$. The proof of the main result is streamlined in Section~\ref{sec:main} in the case of a basis. In Section~\ref{sec:reduction} we show that this case implies the general statement as well. The gist of the proof relies on an application of the Euler-Jacobi vanishing theorem, which requires us to work over the complex numbers: an issue that is resolved in due time as all solutions of the equations we will deal with turn out to be real. In Section~\ref{sec:weakextremal} and Section~\ref{sec:strongextremal}, we prove Theorems~\ref{weakextremal} and Proposition~\ref{strongextremal} on the extremal configurations. Finally, for the reader's sake we have included an Appendix with auxiliary calculus lemmata and simple arguments we will refer to throughout the paper.

% ============================================================
% On the set of local extrema
% ============================================================
\section{On the set of local extrema}\label{sec:extrema}
	In this paper we will define
	$$
	\angles{x, y} = \sum_{j = 1}^d x_jy_j,
	$$
	for any pair of vectors $x,y\in\mathbb{C}^d$. This coincides with the inner product of vectors whenever they are real, which makes our choice of notation self-explanatory, and it is a suitable extension to the complex case. Thus, no complex conjugates are involved. We will abuse notation accordingly. For example, given a function $f:\C^d\to \C^d$, we denote its Jacobian matrix by
	$$
	J_f(x) = \left(\frac{\partial f_j}{\partial x_k}(x)\right)_{j,k = 1,\ldots,d}.
	$$
	Similarly, we refer to the Laplacian as the operator given by the formula
	$$
	\Delta = \sum_{j=1}^d \frac{\partial^2}{\partial x_j^2}.
	$$
	In our application, the functions will be polynomials. When restricted to Euclidean space, these operations clearly agree with the usual definitions.

	Extremal points of $P$ on the sphere are characterized by the identity
	\begin{equation}\label{eq:critical}
		u=\frac{1}{n}\sum_{j=1}^n \frac{v_j}{\langle v_j,u\rangle}
	\end{equation}
	(such a vector is automatically a unit vector, cf. Lemma~\ref{criticalpoints} in the Appendix, for more details).

	\begin{lemma}\label{realpoints}
		Let $v_1,\dots,v_n\in \R^d$. Then every complex solution
		$u\in \mathbb{C}^d$ of equation \eqref{eq:critical} is a real unit vector.
	\end{lemma}
	\begin{proof} By our previous comments, it is enough to prove that it is real.
		Given a solution $u$, none of the denominators in \eqref{eq:critical} can vanish. Let us write $u = a + ib$ with $a,b\in\R^d$. Since each $v_j$ is real, for every $j=1,\ldots,n$ we have
		$$
		\langle v_j,u\rangle
		=
		\langle v_j,a\rangle+i\langle v_j,b\rangle .
		$$
		Furthermore, as none of the denominators in \eqref{eq:critical} vanish, each of these complex numbers is non-zero.
		So (\ref{eq:critical}) becomes
		\begin{equation}\label{fixedx2}
			n(a + ib) = \sum_{j=1}^{n} \frac{\angles{v_j,a}}{\angles{v_j,a}^2 + \angles{v_j,b}^2}v_j -i \sum_{j=1}^n \frac{\angles{v_j,b}}{\angles{v_j,a}^2+\angles{v_j,b}^2}v_j.
		\end{equation}
		Comparing imaginary parts in (\ref{fixedx2}) yields
		\begin{equation}\label{fixedxi}
			nb = - \sum_{j = 1}^n \frac{\angles{v_j,b}}{\angles{v_j,a}^2 + \angles{v_j,b}^2}v_j.
		\end{equation}
		Taking the real inner product with $b$ on both sides of \eqref{fixedxi}, we obtain
		$$
		n\Vert b\Vert^2 = - \sum_{j=1}^n \frac{\angles{v_j,b}^2}{\angles{v_j,a}^2+\angles{v_j,b}^2},
		$$
		which implies that $b=0$.
	\end{proof}

	The following observation is also elementary, a version of which appears in \cite{Leung2008}.

	\begin{lemma}\label{finitepoints}
		Let $v_1,\dots,v_n\in \mathbb R^d\setminus\{0\}$. Then, in each connected component of
		$$
		\mathbb R^d\setminus \bigcup_{j=1}^n
		\{x\in\mathbb R^d:\langle v_j,x\rangle=0\},
		$$
		there is a unique solution to the equation
		\begin{equation}\label{real:eq}
			u=\frac{1}{n}\sum_{j=1}^n
			\frac{v_j}{\langle v_j,u\rangle}.
		\end{equation}

		Moreover, if $n=d$ and $v_1,\dots,v_n$ form a basis of $\mathbb R^n$, then the above complement has exactly $2^n$ connected components.
	\end{lemma}

	\begin{proof}

		The connected components are determined by the sign patterns of the quantities $\langle v_i,x\rangle$. More precisely, on each connected component the functions $\operatorname{sign}\langle v_i,x\rangle$, $i=1,\dots,n$, are constant. Conversely, each nonempty region corresponding to a fixed sign pattern is a
		connected component.

		Let $C$ be any connected component. On $C$, none of the quantities $\langle x,v_j\rangle$ vanish. Define $\Psi:C\to\R$ by
		$$
		\Psi(x) = \frac{1}{2}\norm{x}^2 - \frac{1}{n}\sum_{j=1}^n\log\abs{\angles{v_j,x}}.
		$$
		Its gradient is
		$$
		\nabla \Psi(x) = x-\frac{1}{n}\sum_{j=1}^n \frac{v_j}{\angles{v_j,x}}.
		$$
		Thus, critical points of $\Psi$ are exactly solutions of (\ref{real:eq}).

		It remains to prove existence and uniqueness of the solution. To that end, note that if $x$ approaches the boundary of $C$, then for some $j$, $\langle x,v_j\rangle\to 0,$ and so $-\log\abs{\angles{x,v_j}}\to +\infty.$
		On the other hand,
		$$
		\frac{1}{2}\norm{x}^2-\frac{1}{n}\sum_{j=1}^n\log\abs{\angles{v_j,x}}\to+\infty,\quad \text{as } \norm{x}\to\infty,
		$$
		Thus $\Psi(x)\to+\infty$ at the boundary of $C$ and at infinity.

		Therefore $\Psi$ attains a minimum at some point of $C$. Since the minimum is attained in the interior of $C$, its gradient vanishes there, proving the existence of a solution to (\ref{real:eq}).

		To show uniqueness, we compute the Hessian of $\Psi$:
		$$
		\nabla^2\Psi(x)= I + \frac{1}{n}\sum_{j=1}^n \frac{v_j\otimes v_j}{\angles{v_j,x}^2}.
		$$
		Notice that $\nabla^2\Psi(x)$ is a positive definite matrix for all $x\in\R^d$. Therefore $\Psi$ is strictly convex on $C$ and so it has at most one critical point.

		In the case where $d = n$ and $v_1,\dots,v_n$ form a basis of $\R^n$, the number of connected components is exactly $2^n$. Indeed, for each choice of signs $\varepsilon\in\{-1,1\}^n$, consider the set
		$$
		C_\varepsilon
		=
		\{x\in\R^n: \operatorname{sign}\langle v_i,x\rangle=\varepsilon_i, i=1,\dots,n\}.
		$$
		Since $v_1,\dots,v_n$ are a basis, the linear map $T:\R^n\to\R^n$, defined by
		$$
		T(x)=(\langle v_1,x\rangle,\dots,\langle v_n,x\rangle),
		$$
		is an isomorphism. Therefore each set $C_\varepsilon$ corresponds to the inverse image under $T$ of an open orthant of $\R^n$. Hence each $C_\varepsilon$ is nonempty and connected. Since there are $2^n$ open orthants, the complement has exactly $2^n$ connected components.
	\end{proof}
	
	Let $v_1,\dots,v_n \in \R^n$ be a basis consisting of unit vectors, and let $w_1,\dots,w_n$ be its dual basis, that is,
	$$
	\angles{v_j,w_k} = \delta_{jk},
	$$
	for all $j,k$. Define $h:\C^n \to \C^n$ by
	\begin{equation}\label{eq:h}
		h(x) = \sum_{j=1}^n\left(\angles{v_j,x}\angles{w_j,x}-\frac{1}{n}\right)v_j.
	\end{equation}
	This map will allow us to rewrite the preceding extremal point condition as a polynomial system. Note that the vector-valued function $h$ has polynomial components of degree $2$. Indeed, the $j$th component is given by
		$$
		h_j(x)=\sum_{k=1}^n v_{jk}
		\left(
		\langle v_k,x\rangle\langle w_k,x\rangle-\frac1n
		\right),
		$$
		where $v_k=(v_{1k},\dots,v_{nk})$. Hence $h_j$ is a polynomial of degree at most $2$. Its quadratic part is
		$$
		\sum_{k=1}^n v_{jk}\langle v_k,x\rangle\langle w_k,x\rangle.
		$$
		Since $v_1,\dots,v_n$ is a basis, the matrix $(v_{jk})$ has no zero rows nor columns. Thus, for each $j$, there exists $\ell$ such that $v_{j\ell}\neq 0$. It is enough to evaluate at $x=v_{\ell}$ to check that it is not identically zero:
		$$
		\sum_{k=1}^n v_{jk}\langle v_k,v_\ell\rangle\langle w_k,v_\ell\rangle
		=
		v_{j\ell}\langle v_\ell,v_\ell\rangle
		=
		v_{j\ell}\neq 0,
		$$
		because $\langle w_k,v_\ell\rangle=\delta_{k\ell}$ and $v_\ell$ is a unit vector. Therefore, the quadratic part of $h_j$ is nonzero, so $\deg h_j=2.$ Since $j$ was arbitrary, every component of $h$ has degree exactly $2$.
	
	Notice that $h(x)=0$ is equivalent to
	\begin{equation}\label{eq:hs}
		n\angles{w_j,u}\angles{v_j,u} - 1 = 0,
	\end{equation}
	for any $j=1,\ldots,n$. (Indeed, it suffices to notice that both of these are equivalent to $\langle w_j, h\rangle=0$ for any $j$.)

	We can finally characterize the set of local extrema as an algebraic variety, namely
	\begin{lemma}\label{quadraticeqs}
		The identity $$\mathscr{E}(P) = \{x\in\mathbb{C}^n:h(x)=0\}$$
		holds.
	\end{lemma}
	\begin{proof}
		Notice that $P(u)\neq 0$ is immediately satisfied in both cases. Suppose that $u\in \mathscr{E}(P)$.

		Taking the inner product with respect to $w_k$ in \eqref{eq:critical} yields
		\begin{equation}\label{crith:eq1}
			\angles{w_k,u} = \angles{w_k,\frac{1}{n}\sum_{j = 1}^n \frac{v_j}{\langle v_j,u\rangle}}=\frac{1}{n}\sum_{j = 1}^n \frac{\angles{w_k,v_j}}{\langle v_j,u\rangle} = \frac{1}{n}\sum_{j = 1}^n \frac{\delta_{kj}}{\langle v_j,u\rangle} = \frac{1}{n\langle v_k,u\rangle}.
		\end{equation}
		This is equivalent to $h(u)=0$, by our remarks after the introduction of $h$. The opposite inclusion requires us to suppose that $h(u)=0$. Using that this is equivalent to the first identity of \eqref{crith:eq1} and that the $w_1,\dots,w_n$ also form a basis of $\R^n$, we conclude that $u\in\mathscr{E}(P)$.
	\end{proof}

\section{Proof of Theorem \ref{main}} \label{sec:main}

	In this section we {\em prove the theorem in the case where $v_1,\ldots,v_n\in\mathbb{R}^n$ form a basis}. The general case will follow from this by a limiting process
	(cf. Section~\ref{sec:reduction} for more details). Our proof relies on a result of Jacobi. The one-dimensional case was advanced by Euler, while the two- and three-dimensional cases were explicitly proved by Jacobi. By modern standards, the higher-dimensional versions are a consequence of the theory of residues in several complex variables on projective spaces.

	\begin{lemma}[Euler-Jacobi vanishing theorem]\label{eulerjacobi}
		Let $h=(h_1,\ldots,h_n)$, where each $h_i$ is a polynomial on $\mathbb{C}^n$ of degree $d_i$. Suppose the system
		$$
		h_1=\cdots=h_n=0
		$$
		has exactly $\prod_{i=1}^n d_i$ distinct solutions. Then, for every polynomial $g$ satisfying
		$$
		\deg(g)\leq \sum_{i=1}^n d_i-(n+1),
		$$
		one has
		$$
		\sum_{h(x)=0}\frac{g(x)}{\det J_h(x)}=0,
		$$
		provided all determinants above do not vanish.
	\end{lemma}

	We refer to Griffiths and Harris \cite[p.~671]{Griffiths1994} for a proof.

	In our application $h$ is the one defined in \eqref{eq:h}. It is the use of this result that requires us to virtually work over the complex numbers despite the fact that, ultimately, the zero set turns out to be real. Let us observe that most of the hypotheses already hold. Indeed, each $h_j$ is a quadratic polynomial. Therefore, all the $d_i=2$. Lemmata \ref{realpoints} and \ref{finitepoints} together imply that all the hypotheses are satisfied if we choose $g$ with degree at most $n-1$. Before we continue, it will be convenient to digress and compute the determinant involved at the extremal points (cf. Lemma \ref{quadraticeqs}).

	\begin{lemma}\label{Jacobianh}
		Let $u\in\mathscr{E}(P)$. Then
		$$
		\det J_h(u) = P(u)\det\left(I + \frac{1}{n}\sum_{j=1}^n\frac{v_j\otimes v_j}{\angles{v_j,u}^2}\right).
		$$
	\end{lemma}

	\begin{proof}
		By straightforward differentiation of $h$ (cf. equation \eqref{eq:h}), we obtain
		\begin{align*}
			J_h(u) &= \sum_{j=1}^n ( \angles{v_j,u} v_j\otimes w_j+\angles{w_j,u} v_j\otimes v_j ).
		\end{align*}
		Using Lemma \ref{quadraticeqs} and recalling that $h(u)=0$ is equivalent to equations (\ref{eq:hs}) above, it follows that
		\begin{align*}
			J_h(u) &=\left(\sum_{k=1}^n \angles{v_k,u} v_k\otimes w_k\right)\left(I + \frac{1}{n}\sum_{j=1}^n\frac{1}{\angles{v_j,u}^2}v_j\otimes v_j\right),
		\end{align*}
		where we also used the fact that
		$$
		(v_k\otimes w_k)(v_j\otimes v_j) = \angles{w_k,v_j} v_k\otimes v_j = \delta_{jk} v_k\otimes v_j.
		$$
		On the other hand, the tensor
		$$
		\sum_{j=1}^n \angles{v_j,u} v_j\otimes w_j
		$$
		can be written in the basis $v_1,\dots,v_n$. Since it just scales the $j$-th coordinate of $u$ by a factor $\angles{v_j,u}$, it corresponds to a diagonal matrix whose determinant is $\prod_{j=1}^n \angles{v_j,u} = P(u)$, concluding the proof.

	\end{proof}

	By Lemma~\ref{quadraticeqs} and Lemma~\ref{Jacobianh}, the Euler-Jacobi vanishing theorem \ref{eulerjacobi} can be written as
	\begin{equation}\label{refref}
		\sum_{u\in\mathscr{E}(P)} \frac{g(u)}{P(u)}\mu(u) = 0,
	\end{equation}
	where $\mu:\mathscr{E}(P)\rightarrow(0,\infty)$ is defined by the identity
	$$
	\mu(u)=\det\left(I + \frac{1}{n}\sum_{j=1}^n\frac{v_j\otimes v_j}{\angles{v_j,u}^2}\right)^{-1}>0.
	$$
	(Recall the matrix is positive definite.) Now we take $g(x) = \Delta P(x)$. The polynomial $g(x)$ has degree $n-2$, which satisfies the hypothesis of Theorem \ref{eulerjacobi}. To conclude the proof, we notice that with this choice, by Lemma \ref{lem:phi} one has
	$$
	g(u) = \Delta P(u) = P(u)\left(n^2 - \sum_{j=1}^n\frac{1}{\angles{v_j,u}^2}\right).
	$$
	Our claim follows by plain substitution of this equality in \eqref{refref}.

% ============================================================
% Reduction to the basis case
% ============================================================
\section{Reduction to the basis case}\label{sec:reduction}

	In this section {\em we prove that if Theorem~\ref{main} holds for any basis in $\R^n$ then it also follows for any set of distinct unit vectors such that no pair is parallel}. We first notice that, without loss of generality, we may assume that $v_1,\dots,v_n\in\R^n$. Indeed, if the vectors belong to $\R^d$, depending on whether the dimension $d$ is smaller or larger than $n$, one can pad them with zeros, or replace the ambient space with $\operatorname{span}\{v_1,\cdots,v_n\}$, respectively.

	Let $V = \operatorname{span}\{v_1,\dots,v_n\}$ and let $\ell = \dim V$. Without loss of generality, we may assume that $v_1,\dots,v_\ell$ form a basis of $V$. Let $w_{\ell+1},\dots,w_n$ be an orthonormal basis of $V^\perp$. We introduce a smooth perturbation of the vectors in such a way that, for $t\neq 0$, the perturbed vectors form a basis and satisfy $v_j^t\to v_j$ as $t\to 0$. More precisely, for each $j=1,\dots,n$ and each $t\in(-\pi/2,\pi/2)$, we set
	$$
	v_j^t = \begin{cases}
		v_j, & j\leq \ell\\
		\cos(t) v_j + \sin(t) \, w_j, & j>\ell
	\end{cases}
	$$
	Then, clearly $t\mapsto v_j^t$ is a smooth function for all $j = 1,\dots,n$ and $v_1^t,\dots,v_n^t$ is a basis of $\R^n$ for all $t\in (-\pi/2,\pi/2)\setminus \{0\}$. It is natural to define the polynomial associated with this:
	$$
	P_t(x) = \prod_{j=1}^n\langle v_j^t,x\rangle
	$$
	and denote its set of local extremizers by $\mathscr{E}(P_t)\subseteq \Sp^{n-1}$. Note that by Lemma~\ref{finitepoints}, $\mathscr{E}(P_t)$ has exactly $2^n$ elements which can be described as the set of unit vectors satisfying
	$$
	u = \frac{1}{n}\sum_{j=1}^n\frac{v_j^t}{\langle v_j^t,u\rangle}.
	$$
	Define the function on $\R^n$
	$$
	F(x,t) = x - \frac{1}{n}\sum_{j=1}^n\frac{v_j^t}{\langle v_j^t,x\rangle},
	$$
	whenever $\langle v_j^t,x\rangle \neq 0$ for all $j = 1,\dots,n$. Notice that for every $u\in \mathscr{E}(P)$, we have that
	$$
	F(u,0) = 0.
	$$
	Moreover, $F$ is continuously differentiable in a neighbourhood of $(u,0)$ and its differential with respect to $x$ is
	$$
	dF_x(x,t) = I + \frac{1}{n}\sum_{j=1}^n\frac{v_j^t\otimes v_j^t}{\langle v_j^t,x\rangle^2}.
	$$
	At $(u,0)$, with $u\in\mathscr{E}(P)$, this differential is positive definite,
	hence invertible. By the implicit function theorem applied at $t=0$, for every $u \in \mathscr{E}(P)$, there exist neighbourhoods $U_u$ of $u$ and $I_u = (-\delta_u,\delta_u)$ of $t=0$, and a unique smooth function $t\mapsto x(u,t)$ defined on $I_u$, such that $x(u,0) = u$, and
	$$
	F(x(u,t),t) = 0.
	$$
	Additionally, this solution is unique among $x\in U_u$. Since $\mathscr{E}(P)$ is finite, the neighbourhood $U_u$ may be chosen to be pairwise disjoint and the constant $\delta_u>0$ may be chosen independent of $u$. In other words, there exists $\delta>0$ such that for every $u\in\mathscr{E}(P)$ and every $t\in[0,\delta)$, there is a unique $x(u,t)\in U_u\cap \mathscr{E}(P_t)$ with $x(u,0)=u$.

	Let us define the set
	$$
	\mathscr{B}_t = \{x(u,t): u\in \mathscr{E}(P)\}
	$$
	and its complement $\mathscr{M}_t = \mathscr{E}(P_t)\setminus \mathscr{B}_t$ for $t\in(0,\delta)$. Additionally, for all $t\in (0,\delta)$ and for all $u\in\mathscr{E}(P_t)$, let
	$$
	\mu_t(u) = \det\left(I + \frac{1}{n}\sum_{j=1}^n\frac{v_j^t\otimes v_j^t}{\langle v_j^t, u\rangle^2}\right)^{-1}.
	$$
	These will allow us to study the limit $t\rightarrow 0$ by splitting the sum involved into two, namely
	$$
	\sum_{u^t\in \mathscr{E}(P_t)}\left(\sum_{j=1}^n\frac{1}{\langle v_j^t, u^t\rangle^2}-n^2\right)\mu_{t}(u^t) = B_t + M_t,
	$$
	where
	$$
	B_t = \sum_{u^t\in \mathscr{B}_t}\left(\sum_{j=1}^n\frac{1}{\langle v_j^t, u^t\rangle^2}-n^2\right)\mu_t(u^t)
	$$
	and
	$$
	M_t = \sum_{u^t\in \mathscr{M}_t}\left(\sum_{j=1}^n\frac{1}{\langle v_j^t, u^t\rangle^2}-n^2\right)\mu_{t}(u^t).
	$$
	
	We claim the following auxiliary lemma.
	
	\begin{lemma}\label{limitsBM}
		Under the hypotheses of Theorem \ref{main}, the following limits hold:
		\begin{enumerate}[label=$(\alph*)$]
			\item \label{limitB}
			$
			\displaystyle\lim_{t\to 0}B_t
			=
			\sum_{u\in \mathscr{E}(P)}
			\left(
			\sum_{j=1}^n\frac{1}{\langle v_j,u\rangle^2}
			-n^2
			\right)\mu(u).
			$
			\item\label{limitM} $\displaystyle\lim_{t\to 0}M_t=0.$
		\end{enumerate}
	\end{lemma}
	These will be proved in the next subsection. For the time being let us observe that assuming Lemma \ref{limitsBM} we get
	$$
	\lim_{t\to 0} \sum_{u^t\in \mathscr{E}(P_t)}\left(\sum_{j=1}^n\frac{1}{\langle v_j^t, u^t\rangle^2}-n^2\right)\mu_{t}(u^t) = \lim_{t\to 0} (B_t + M_t) = \sum_{u\in \mathscr{E}(P)}\left(\sum_{j=1}^n\frac{1}{\langle v_j, u\rangle^2}-n^2\right)\mu(u),
	$$
	which is zero by the basis case of Theorem \ref{main}, as proved in Section \ref{sec:main} above.

% ------------------------------------------------------------
% Proof of Lemma \ref{limitsBM}
% ------------------------------------------------------------
\subsection{Proof of Lemma \ref{limitsBM}}

	For any $u^t\in \mathscr{B}_t$ we have that $u^t = x(u,t)$ for some $u\in \mathscr{E}(P)$. By continuity, one gets
	$$
	\left(\sum_{j=1}^n\frac{1}{\langle v_j^t, u^t\rangle^2}-n^2\right)\mu_t(u^t)\to \left(\sum_{j=1}^n\frac{1}{\langle v_j, u\rangle^2}-n^2\right)\mu(u).
	$$
	This proves Lemma~\ref{limitsBM}\ref{limitB}.

	Let us turn our attention to the local extrema on the points $u^t\in \mathscr{M}_t$. For these we prove the following:

	\begin{lemma}\label{lem:badpoints}
		Let $t_m\to 0$, and let $u^{t_m}\in \mathscr{M}_{t_m}$ be a convergent sequence. 
		Then, there exist distinct indices $j\neq k$ such that $\langle v_j^{t_m},u^{t_m}\rangle\to 0$ and $\langle v_k^{t_m},u^{t_m}\rangle\to 0.$
	\end{lemma}
	
	\begin{proof} Without loss of generality, we may assume that
		$$
		u^t \to u \in \Sp^{n-1}.
		$$
		We will prove that there exists $j\neq k$ such that $\angles{v_j,u} = \angles{v_k,u} = 0$.

		By construction, $u\notin\mathscr{E}(P)$. Indeed, if $u\in\mathscr{E}(P)$, then for sufficiently small $t>0$, $u^t\in U_u$. By the uniqueness given by the implicit function theorem, $u^t = x(u,t)\in \mathscr{B}_t$, contradicting the fact that $u^t\in \mathscr{M}_t$. Since $u^t\in\mathscr{M}_t\subseteq\mathscr{E}(P_t)$, it is a critical point of $P_t$ on $\Sp^{n-1}$
		$$
		nP_t(u^t) u^t=\nabla P_t(u^t) .
		$$
		Taking the limit as $t\to 0$ yields
		\begin{equation}\label{vanishinggrad}
			nP(u) u=\nabla P(u).
		\end{equation}
		If $P(u) \neq 0$, then
		$$
		u = \frac{1}{n}\sum_{j = 1}^n\frac{v_j}{\angles{v_j,u}},
		$$
		so $u\in\mathscr{E}(P)$, a contradiction. Since $P(u)=0$ necessarily, by \eqref{vanishinggrad}, it follows that $\nabla P(u)=0$. For a product of linear forms, this implies that at least two factors of the polynomial $P$ vanish at $u$, proving Lemma \ref{lem:badpoints}.
	\end{proof}
    
	This will allow us to show that, given a sequence $t_m\to 0$ and a convergent sequence
	$u^{t_m}\in \mathscr{M}_{t_m}$, the limit of the corresponding summand in
	Lemma \ref{limitsBM}\ref{limitM} vanishes. Indeed, by Lemma \ref{lem:lowerdet},
	equation \eqref{det:inq} below, we have
	$\mu_{t_m}(u^{t_m})\leq D_{t_m}^{-1}$, where
	$$
	D_{t_m} = 1 + \frac{1}{n}\sum_{j=1}^n \frac{1}{\langle v_j^{t_m},u^{t_m}\rangle^2} + \frac{1}{n^2}\sum_{1\leq j<k\leq n} \frac{\varepsilon_0}{\langle v_j^{t_m},u^{t_m}\rangle^2\langle v_k^{t_m},u^{t_m}\rangle^2},
	$$
	and  $\varepsilon_0>0$ is fixed and chosen so that
	$$
	\min_{1\leq j<k\leq n}\sin^2\theta^{t_m}_{jk}>\varepsilon_0>0
	$$
    for any $t_m$.
	This is possible because
	$\sin^2\theta^{t_m}_{jk}\to \sin^2\theta_{jk}>0$ for all $j\neq k$,
	using the hypothesis that no pair of limiting vectors is parallel.
	As a consequence,
	$$\abs{\sum_{j=1}^n\frac{1}{\langle v_j^{t_m},u^{t_m}\rangle^2} - n^2} \mu_{t_m}(u^{t_m}) \leq \sum_{j=1}^n \frac{1}{\langle v_j^{t_m},u^{t_m}\rangle^2D_{t_m}} +
	\frac{n^2}{D_{t_m}}.
	$$
	By Lemma~\ref{lem:badpoints}, there exist distinct indices $j\neq k$ such that
	$$
	\langle v_j^{t_m},u^{t_m}\rangle\to 0 \quad \text{and} \quad \langle v_k^{t_m},u^{t_m}\rangle\to 0.$$	
	Therefore,
	$$
	D_{t_m} \geq \frac{\varepsilon_0}{n^2 \langle v_j^{t_m},u^{t_m}\rangle^2 \langle v_k^{t_m},u^{t_m}\rangle^2} \to \infty,
	$$
	and the second summand above vanishes in the limit, i.e. $n^2/D_{t_m}\to 0$.
	
	Similarly we can control the sum that preceeds it. Fix $j = 1,\ldots n$. Notice that the limit $$\langle v_k^{t_m},u^{t_m}\rangle\to 0$$  implies that
	$$
	\frac{1}
	{\langle v_j^{t_m},u^{t_m}\rangle^2D_{t_m}}
	\leq
	\frac{n^2}{\varepsilon_0}
	\langle v_k^{t_m},u^{t_m}\rangle^2
	\to 0
	$$
	also vanishes in the limit. Since $j$ was arbitrary, we obtain
	$$
	\sum_{j=1}^n
	\frac{1}
	{\langle v_j^{t_m},u^{t_m}\rangle^2D_{t_m}}
	\to 0.
	$$
	Consequently,
	$$
	\abs{\sum_{j=1}^n \frac{1}{\langle v_j^{t_m},u^{t_m}\rangle^2} - n^2}\mu_{t_m}(u^{t_m})\to 0.
	$$
	Thus, along a convergent subsequence, the corresponding
	summand in Lemma \ref{limitsBM}\ref{limitM} vanishes.

	We now conclude that $M_t\to 0$. Suppose, by contradiction, that this is not
	the case. Then there exist $\eta>0$ and a sequence $t_m\to 0$ such that
	$$
	\abs{M_{t_m}}\geq \eta
	$$
	for every $m$. Since $\#\mathscr{E}(P_{t_m})=2^n$, the number of points in
	$\mathscr{M}_{t_m}$ is uniformly bounded. Passing to a subsequence, we may
	label the points of $\mathscr{M}_{t_m}$ as
	$$
	u^{t_m}_1,\ldots,u^{t_m}_N,
	$$
	with $N\leq 2^n$ independent of $m$, in such a way that each sequence $u^{t_m}_r$ converges on $\mathbb{S}^{n-1}$. By the above argument, for
	each fixed $r=1,\ldots,N$, the corresponding summand tends to zero. Since
	there are only finitely many such summands, we obtain
	$$
	M_{t_m}\to0,
	$$
	contradicting $\abs{M_{t_m}}\geq\eta$. Hence $M_t\to 0$, proving Lemma~\ref{limitsBM}\ref{limitM}. It remains only to prove the determinant estimate used above.

	\begin{lemma}\label{lem:lowerdet}
		Let $v_1,\dots,v_n \in \R^n$ be unit vectors, and let $\theta_{jk}$ be the angles between $v_j$ and $v_k$. Let $u\in \Sp^{n-1}$ be such that $\angles{v_j,u}\neq 0$ for all $j$. Then
		\begin{equation}\label{det:inq}
			\det\left(I + \frac{1}{n}\sum_{j=1}^n\frac{v_j\otimes v_j}{\angles{v_j,u}^2}\right)\geq 1 + \frac{1}{n}\sum_{j=1}^n\frac{1}{\angles{v_j,u}^2} + \frac{1}{n^2}\sum_{1\leq j<k\leq n}\frac{\sin^2\theta_{jk}}{\angles{v_j,u}^2\angles{v_k,u}^2}
		\end{equation}
		holds.
	\end{lemma}
	\begin{proof}
		Set
		$$
		a_j = \frac{1}{\sqrt{n}}\frac{v_j}{\angles{v_j,u}},
		$$
		for $j = 1,\dots,n$ and define
		$$
		A = \sum_{j=1}^n a_j\otimes a_j.
		$$
		Since each $a_j\otimes a_j$ is positive semidefinite, the matrix $A$ is positive semidefinite. Hence, its eigenvalues $\lambda_1,\dots,\lambda_n$ are all nonnegative.

		We expand
		$$
		\det\left(I + A\right) = \prod_{j = 1}^n(1+\lambda_j) = 1 + \sum_r \lambda_r + \sum_{r<s} \lambda_r\lambda_s + R
		$$
		where $R$ is the sum of all terms involving products of at least three eigenvalues. Since all $\lambda_r\geq 0$, every term in $R$ is nonnegative, and so is $R$.

		We now compute the first-order term. Using that $v_j$ is a unit vector for all $j=1,\ldots,n$, we get
		$$
		\sum_{r=1}^n \lambda_r = \tr(A) = \sum_{j=1}^n \norm{a_j}^2 = \frac{1}{n}\sum_{j=1}^n\frac{1}{\angles{v_j,u}^2}.
		$$
		Next we compute the second-order term. We get
		$$
		\sum_{1\leq r < s\leq n} \lambda_r\lambda_s = \frac{1}{2}\left[\left(\sum_{r=1}^n \lambda_r\right)^2 - \sum_{r=1}^n \lambda_r^2\right]= \frac{\tr(A)^2 - \tr(A^2)}{2}.
		$$
		Now
		$$
		\tr(A)^2 = \sum_{j,k=1}^n\norm{a_j}^2 \norm{a_k}^2.
		$$
		On the other hand, using
		$$
		A^2 = \sum_{j,k=1}^n \angles{a_j,a_k} a_j\otimes a_k
		$$
		we get
		$$
		\tr(A^2) = \sum_{j,k=1}^n \angles{a_j,a_k}^2.
		$$
		Thus, if $\theta_{jk}$ is the angle between $a_j$ and $a_k$,
		$$
		\tr(A)^2 - \tr(A^2) = \sum_{j,k=1}^n \left(\norm{a_j}^2 \norm{a_k}^2-\angles{a_j,a_k}^2\right) = \sum_{j,k=1}^n \norm{a_j}^2 \norm{a_k}^2\sin^2\theta_{jk}.
		$$
		Finally, by the definition of $a_j$, we obtain
		$$
		\norm{a_j}^2 \norm{a_k}^2\sin^2\theta_{jk} = \frac{1}{n^2}\frac{\sin^2\theta_{jk}}{\angles{v_j,u}^2\angles{v_k,u}^2}.
		$$
		Combining all our previous identities and taking into account that $R\geq 0$ gives the desired inequality.
	\end{proof}
	
% ============================================================
% Proof of Theorem~\ref{weakextremal}
% ============================================================
\section{Proof of Theorem~\ref{weakextremal}}\label{sec:weakextremal}

	First, we see that the factors of $P$ in an extremal configuration must be simple (cf. Lemma \ref{extremalreduction} from the Appendix). Suppose that $v_1,\dots,v_n\in \Sp^{d-1}$ is an extremal configuration of Conjecture~\ref{weak}. By Lemma~\ref{extremalreduction}, no two of the $v_j$'s are parallel. For all $u\in \Sp^{d-1}$ with $P(u)\neq 0$, the Arithmetic--Geometric Mean inequality gives
	\begin{equation}\label{AMGM}
		n = \inf_{v\in \Sp^{d-1}}{\abs{P(v)}^{-\frac{2}{n}}} \leq \abs{P(u)}^{-\frac{2}{n}} = \left(\prod_{j=1}^n\frac{1}{\angles{v_j,u}^2}\right)^{\frac{1}{n}} \leq \frac{1}{n}\sum_{j=1}^n\frac{1}{\angles{v_j,u}^2}
	\end{equation}
	for all $u\in \Sp^{d-1}$ with $P(u)\neq0$. This implies that for all $u\in \Sp^{d-1}$ with $P(u)\neq0$, we have that
	$$
	\sum_{j=1}^n\frac{1}{\angles{v_j,u}^2}-n^2 \geq 0.
	$$
	Therefore, by Theorem~\ref{main}, for all $u\in\mathscr{E}(P)$ the identity
	$$
	\sum_{j=1}^n\frac{1}{\angles{v_j,u}^2}-n^2 = 0
	$$
	necessarily holds. Hence, for $u\in\mathscr{E}(P)$ equality holds throughout the chain of inequalities in \eqref{AMGM}. In particular, equality holds in the Arithmetic-Geometric Mean inequality above, which implies that
	$$
	\angles{v_1,u}^2 = \cdots = \angles{v_n,u}^2 = \frac{1}{n}.
	$$
	As a consequence, substituting in \eqref{eq:critical} yields, for all $u\in \mathscr{E}(P)$,
	\begin{equation}\label{bangtype}
		u = \frac{1}{\sqrt{n}}\sum_{j = 1}^n \varepsilon_j v_j
	\end{equation}
	for some $\varepsilon \in\{-1,1\}^n$. Hence, there exists a subset $S\subset\{-1,1\}^n$ such that
	$$
	\mathscr{E}(P) = \left\{u_\varepsilon = \frac{1}{\sqrt{n}}\sum_{j = 1}^n \varepsilon_j v_j: \varepsilon\in S\right\}.
	$$
	We now show that $v_1,\dots,v_n$ are orthogonal. Taking the inner product with $v_k$ in \eqref{bangtype} yields
	$$
	\frac{\varepsilon_k}{\sqrt{n}} = \angles{v_k,u} = \frac{1}{\sqrt{n}}\sum_{j = 1}^n \varepsilon_j \angles{v_k,v_j}.
	$$
	Equivalently, for every $\varepsilon\in S$ we have
	$$
	G\varepsilon = \varepsilon,
	$$
	where $G = (\angles{v_j,v_k})_{j,k=1}^n$ is the Gram matrix of $v_1,\dots,v_n$. Thus, if we show that the elements $\varepsilon\in S$ span $\R^n$, then $G = I$, and hence $v_1,\dots,v_n$ are orthogonal. Once this is proven, it follows from the orthogonality that $d\geq n$.

	To this end, fix $j = 1,\ldots,n$. There exists $x\in v_j^\perp$ such that $x\notin v_k^\perp$ for any $k\neq j$. Indeed, since no two of the vectors $v_k$ are parallel, the hyperplanes $v_k^\perp$ are pairwise distinct. Therefore, the hyperplane $v_j^\perp$ is not contained in the union of the remaining hyperplanes. This proves the existence of such an $x$. As a consequence, in a neighbourhood of $x$ there are two adjacent connected components $C$ and $C'$, separated by the hyperplane $v_j^\perp$. On each connected component, the sign vector
	$$
	(\operatorname{sign}\angles{v_1,u},\dots, \operatorname{sign}\angles{v_n,u})
	$$
	is constant (cf. Lemma \ref{finitepoints}). Let $\varepsilon$ and $\varepsilon'$ be the sign vectors of $C$ and $C'$, respectively. By continuity, in a sufficiently small neighbourhood of $x$, the signs of $\angles{v_k,u}$ remain constant for all $k\neq j$. However, crossing the hyperplane $v_j^\perp$ changes the sign of $\angles{v_j,u}$. In other words, $\varepsilon'_j = -\varepsilon_j$, while $\varepsilon'_k = \varepsilon_k$ for every $k\neq j$. By Lemma \ref{finitepoints}, each of these connected components contains a local extremum $u\in \mathscr{E}(P)$. This shows that $\varepsilon',\varepsilon\in S$ and, furthermore,
	$$
	e_j = \frac{\varepsilon-\varepsilon'}{2\varepsilon_j}\in \operatorname{span}(S).
	$$
	This holds for $j = 1,\dots,n$, proving that all the standard basis vectors $e_1,\dots,e_n$ belong to $\operatorname{span}(S)$. Consequently, $\operatorname{span}(S) = \R^n$, finishing the proof.

% ============================================================
% Proof of Proposition~\ref{strongextremal}
% ============================================================
\section{Proof of Proposition~\ref{strongextremal}}\label{sec:strongextremal}

	To prove extremality, we show that
	$$
	\sum_{j=1}^n\frac{1}{\angles{v_j,u}^2}\geq n^2
	$$
	for all $u\in \Sp^{d-1}$ with $P(u)\neq 0$.
	Indeed, by our hypothesis and Lemma \ref{lem:L=0} we know that $P$ is harmonic. The identity $\Delta P\equiv 0$ together with the identity \eqref{Laplacian} from Lemma \ref{lem:grad} shows that
	$$
	\sum_{j=1}^n\frac{1}{\angles{v_j,u}^2} = \norm{\sum_{j=1}^n\frac{v_j}{\langle v_j,u\rangle}}^2
	$$
	holds for all $u\in \Sp^{d-1}$ with $P(u)\neq 0$. Applying Cauchy--Schwarz, one gets
	$$
	n^2 = \angles{u,\sum_{j=1}^n\frac{v_j}{\langle v_j,u\rangle}}^2\leq \norm{\sum_{j=1}^n\frac{v_j}{\langle v_j,u\rangle}}^2\norm{u}^2 = \norm{\sum_{j=1}^n\frac{v_j}{\langle v_j,u\rangle}}^2,
	$$
	with equality precisely when $u\in\mathscr{E}(P)$.

% ============================================================
% Appendix
% ============================================================
\section{Appendix}

	In this section we summarize a number of observations including straightforward, yet subtle, identities satisfied by $P$ and its critical points. The following observation, which was certainly known to Ball and Frenkel, explains the elementary link between the stronger statement and the original polarization problem.

	\begin{proposition}\label{lem:AGM}
		Conjecture~\ref{strong} implies Conjecture~\ref{weak}.
	\end{proposition}
	\begin{proof}
		By the Arithmetic-Geometric Mean inequality, for all $x\in \R^d$ with $P(x)\neq0$, we obtain
		\begin{equation}
			\abs{P(x)}^{-\frac{2}{n}} = \left(\prod_{j=1}^n\frac{1}{\angles{v_j,x}^2}\right)^{\frac{1}{n}} \leq \frac{1}{n}\sum_{j=1}^n\frac{1}{\angles{v_j,x}^2}.
		\end{equation}
		The hypothesis then implies that there exists $u\in \Sp^{d-1}$ such that $\abs{P(u)}^{-\frac{2}{n}}\leq n$ which is equivalent to Conjecture~\ref{weak}.
	\end{proof}

	The following is known in the literature (cf. \cite{Steinberg1964} and \cite[Theorem~6.1]{Agranovsky2000}); nevertheless, we include a proof here for completeness.

	\begin{lemma}\label{lem:L=0}
		Given a finite reflection system as in Proposition~\ref{strongextremal}, the associated polynomial $P$ is harmonic.
	\end{lemma}
	\begin{proof}
		Let $v_1,\dots,v_n\in \Sp^{d-1}$ be unit vectors, no two of which are parallel, and suppose that $\Phi=\{\pm v_1,\dots,\pm v_n\}$ is a finite reflection system. For each $k$, the reflection $s_{v_k}$ preserves the reflection system $\Phi$. Hence $P\circ s_{v_k}$ has the same set of linear factors as $P$, possibly up to a sign. Therefore, there exists some $\varepsilon\in\{-1,1\}$ such that
		$$
		P\circ s_{v_k}=\varepsilon P.
		$$
		We claim that $\varepsilon=-1$. Choose $x\in v_k^\perp$ with $\angles{v_j,x}\neq 0$ for all $j\neq k$. Define
		$$
		f(t)= P(x + tv_k),
		$$
		for all $t\in \R$. Since $x\in v_k^\perp$ and $\norm{v_k}=1$, we have $\angles{v_k , x + t v_k}=t.$ Thus
		$$
		f(t)=t\prod_{j\neq k}(\angles{v_j,x} + t\angles{v_j,v_k}).
		$$
		On the other hand, since $s_{v_k}(x+t v_k)=x-t v_k$, the identity $P\circ s_{v_k}=\varepsilon P$ implies that $f(-t)=\varepsilon f(t)$. Differentiating with respect to $t$, we obtain $-f'(-t)=\varepsilon f'(t).$ Substituting $t=0$ yields $-f'(0)=\varepsilon f'(0).$ Since, by our choice of $x$, we know that
		$$
		f'(0)=\prod_{j\neq k}\angles{v_j,x}\neq 0.
		$$
		It follows that $\varepsilon=-1$.

		To conclude, we take any $x\in v_k^\perp$. Then
		$$
		\Delta P(x) = \Delta P(s_{v_k}(x)) = \Delta (P\circ s_{v_k})(x)= - \Delta P(x).
		$$
		This implies that $\Delta P(x)=0$ for all $x\in v_k^\perp$. It is elementary to show, using a rotation and a representation of the polynomial $\Delta P$ if necessary, that the linear form $\langle v_k,x\rangle$ divides $\Delta P(x)$ for every $k=1,\dots,n$. Since the factors are coprime, it follows that $P$ divides $\Delta P$. However, $\Delta P$ has degree $n-2$, whereas $P$ has degree $n$. Hence this is only possible if $\Delta P\equiv 0$.
	\end{proof}

	\begin{lemma} \label{lem:grad}
		For any $x\in \R^d$ with $P(x)\neq 0$, the gradient of $P$ is
		\begin{equation}\label{gradP}
			\nabla P(x) = P(x)\sum_{j=1}^n\frac{v_j}{\angles{v_j,x}},
		\end{equation}
		and the Laplacian is
		\begin{equation}\label{Laplacian}
			\Delta P(x) = P(x)\left(\norm{\sum_{j=1}^n\frac{v_j}{\langle v_j,x\rangle}}^2-\sum_{j=1}^n\frac{1}{\langle v_j,x\rangle^2}\right).
		\end{equation}
	\end{lemma}
	\begin{proof}
		To compute the gradient take $x\in \C^d$ with $P(x) \neq 0$. By the product rule,
		\begin{equation*}
			\nabla P(x) = \sum_{j=1}^{n}\prod_{k\neq j}\angles{v_k,x}v_j
			= \sum_{j=1}^{n}\frac{P(x)}{\angles{v_j,x}}v_j
			= P(x)\sum_{j=1}^{n}\frac{v_j}{\angles{v_j,x}}.
		\end{equation*}
		Here, in the second and third equalities we use that $\angles{v_j,x}\neq 0$ for all $j$. To compute the Laplacian we first differentiate both sides of (\ref{gradP}) and apply the product rule again
		\begin{align*}
			\nabla^2P(x) &=  \sum_{j=1}^n\frac{v_j}{\angles{v_j,x}}\otimes \nabla P(x) - P(x) \sum_{j=1}^{n}\frac{v_j\otimes v_j}{\angles{v_j,x}^2}\\
			&= P(x) \left(\sum_{j=1}^n\frac{v_j}{\angles{v_j,x}}\otimes\sum_{j=1}^n\frac{v_j}{\angles{v_j,x}} - \sum_{j=1}^{n}\frac{v_j\otimes v_j}{\angles{v_j,x}^2}\right).
		\end{align*}
		The trace of the above yields,
		\begin{equation*}
			\Delta P(x) = P(x)\left(\norm{\sum_{j=1}^n\frac{v_j}{\langle v_j,x\rangle}}^2-\sum_{j=1}^n\frac{1}{\langle v_j,x\rangle^2}\right).
		\end{equation*}
	\end{proof}

	   The following observation characterizes the extremal points. It follows from a straightforward application of Lagrange multipliers and, to the best of our knowledge, first appeared as Lemma 13 in \cite{Revesz2004}.

	\begin{lemma}\label{criticalpoints}
		Consider $P(x)$ as a polynomial on $\R^d$. A point $u$ belongs to $\mathscr{E}(P)$ if and only if
		\begin{equation}\label{CPequ}
			u = \frac{1}{n}\sum_{j = 1}^n \frac{v_j}{\langle v_j,u\rangle}.
		\end{equation}
		Furthermore, any vector $u$ satisfying this identity is a unit vector.
	\end{lemma}

	\begin{proof} It is evident that no denominator can vanish. As a consequence, $P(u)\neq 0$.
		A vector $u\in \mathbb{S}^{d-1}$ is a critical point of $P$ on the sphere if and only if $u$ is proportional to the gradient of $P$. If that is the case, there exists $\lambda\in\R$ such
		that
		$$
		u = \lambda \nabla P(u).
		$$
		Using the formula for the gradient (\ref{gradP}), it follows that
		\begin{equation}\label{cplemma:eq1}
			u = \lambda P(u) \sum_{j=1}^n\frac{v_j}{\angles{v_j,u}}.
		\end{equation}
		The reverse implication is even easier as \eqref{cplemma:eq1} already implies $u$ and $\nabla P(u)$ are homothetic by the same reasoning.
		Finally, taking the inner product with $u$ on both sides of \eqref{cplemma:eq1} yields
		$$
		1 = \Vert u\Vert^2 = \lambda P(u) n
		$$ and so $\lambda = \frac{1}{P(u)n}$. Substituting the value of $\lambda$ back in (\ref{cplemma:eq1}), we obtain the desired identity.
	\end{proof}

	\begin{lemma}\label{lem:phi}
		For any $u\in \mathscr{E}(P)$, the Laplacian satisfies the identity
		\begin{equation}\label{LaplacianCP}
			\Delta P(u) = P(u)\left(n^2-\sum_{j=1}^n\frac{1}{\angles{v_j,u}^2}\right).
		\end{equation}
	\end{lemma}
	\begin{proof}
		Suppose that $u\in \mathscr{E}(P)$. It follows from (\ref{CPequ}) that
		$$
		\norm{\sum_{j=1}^n\frac{v_j}{\langle v_j,u\rangle}}^2 = \Vert nu\Vert^2 = n^2.
		$$
		The desired formula follows now by a substitution of the above in (\ref{Laplacian}).
	\end{proof}

	\begin{lemma}\label{extremalreduction}
		An extremal configuration of Conjecture~\ref{weak} must have a square-free associated polynomial $P$.
	\end{lemma}

	\begin{proof} We proceed by contradiction. Suppose that the factor $\angles{v_1,x}$ appears with multiplicity greater than or equal to two in $P$. Then,
		$$
		P(x) = \angles{v_1,x}^2 Q(x),
		$$
		where the polynomial $Q$ is the product of the rest of the factors in $P$. Take a unit vector $w \in v_1^\perp$ and define
		$$
		v_\theta = \cos\theta\,v_1 + \sin\theta\, w,
		$$
		for $\theta\in (-\pi/2,\pi/2)$. Note that if $\pi/2>\theta>0$ is sufficiently small, neither $v_{\theta}$ nor $v_{-\theta}$ can be parallel to any of the original vectors.
		On the other hand, the associated polynomial
		$$
		P_\theta(x) = \angles{v_\theta,x} \angles{v_{-\theta},x} Q(x) = \left(\cos^2\theta \angles{v_1,x}^2 - \sin^2\theta \angles{w,x}^2\right)Q(x).
		$$
		Using that $|\langle w,u\rangle|\leq 1$ and $|Q(u)|\leq 1$, it is straightforward to infer the inequality
		$$
		\abs{P_\theta(u)} \leq \max\{\cos^2\theta |P(u)|,\sin^2\theta\}.
		$$
		Taking $\theta$ small enough and taking into account that $|\cos(\theta)|< 1$ one gets
		$$
		\sup_{u\in\Sp^{d-1}}{\abs{P_\theta(u)}}<\sup_{u\in\Sp^{d-1}}{\abs{P(u)}}.
		$$
		Applying the same procedure iteratively we produce a polynomial $P_\theta$ with no double factor and satisfying the same strict inequality. This contradicts the extremality of the vectors defining the polynomial $P$. 
	\end{proof}

% ============================================================
% Acknowledgements
% ============================================================
\section*{Acknowledgements}

	The second author is deeply grateful to Professor Keith M. Ball for introducing him to this problem, and for his guidance and many helpful conversations during his PhD studies at the University of Warwick. The second author also thanks Gergely Ambrus, Leo Brauner, Daniel Galicer, and Damian Pinasco for valuable discussions. 
    
    The authors are indebted to Professor Keith M. Ball for several comments and suggestions that helped improve the presentation of the paper. They are also grateful to Ujué Etayo for her encouragement, and to Yufei Zhao for pointing out how to fix a minor gap in the proof of Lemma \ref{extremalreduction}.
	
	During the research and writing process leading to the completion of this paper, the authors made use of a large language model. These interactions led to a number of interesting computational insights and minor edits. The paper includes a figure drawn with the aid of a Python program entirely written by artificial intelligence.

	\bibliographystyle{abbrv}
	\bibliography{references}
	\vspace{.5em}

	\noindent
	\begin{minipage}[t]{0.48\textwidth}
		\noindent
		Ángel D. Martínez\\
		\textsc{Departamento de Matemáticas}\\
		\textsc{CUNEF Universidad}\\
		\textsc{Madrid, Spain}\\
		\textit{Email:} \texttt{angeld.martinez@cunef.edu}
	\end{minipage}
	\hfill
	\begin{minipage}[t]{0.48\textwidth}
		\noindent
		Oscar Ortega-Moreno\\
		\textsc{Departamento de Matemáticas}\\
		\textsc{CUNEF Universidad}\\
		\textsc{Madrid, Spain}\\
		\textit{Email:} \texttt{oscar.ortegamoreno@cunef.edu}
	\end{minipage}

\end{document}